%% file: nhmd.tex
\documentclass[reqno]{amsart}
\input{definitions.tex}
\excludecomment{arma}
\gettimestamp Time-stamp: <2017-07-09 23:36:56>

\begin{document}

\title[Invariant tori]{Nos{é}-Thermostated Mechanical Systems on the $n$-Torus}
\author{Leo T. Butler}
\address{Department of Mathematics, University of Manitoba,
Winnipeg, MB, Canada, R3T 2N2}
\email{leo.butler@umanitoba.ca}
\date{\timestamp}
\thanks{The author thanks an anonymous referee for the valuable
  suggestions provided.}

\begin{abstract}
  Let $H(q,p) = ½ ⎸ p ⎹ ² + V(q)$ be an $n$-degree of freedom $C^r$
  mechanical Hamiltonian on $\cotangent \T^n$ where $r>2n+2$. When the
  metric $⎸ · ⎹$ is flat, the Nos{é}-thermostated system associated to
  $H$ is shown to have a positive-measure set of invariant tori near
  the infinite temperature limit. This is shown to be true for all
  variable mass thermostats similar to Nos{é}'s, too. These results
  complement results of Legoll, Luskin \& Moeckel and the
  author~\cite{MR2299758,MR2519685,arxiv-pp1}.
\end{abstract}
\begin{arxivabstract}
  Let H(q,p) = ½ | p | ² + V(q) be an n-degree of freedom C^r
  mechanical Hamiltonian on the cotangent bundle of the n-torus where
  r>2n+2. When the metric | · | is flat, the Nosé-thermostated system
  associated to H is shown to have a positive-measure set of invariant
  tori near the infinite temperature limit. This is shown to be true
  for all variable mass thermostats similar to Nosé's, too. These
  results complement results of Legoll, Luskin & Moeckel and the
  author.
\end{arxivabstract}

\maketitle

\section{Introduction} \label{sec:intro}

In equilibrium statistical mechanics, the mechanical Hamiltonian
$H(q,p)$ is the internal energy of an infinitesimal system $S$ that is
immersed in a heat bath $B$ at the temperature $T$. Nos{é}
\cite{nose}, based on earlier work of Andersen \cite{andersen}, introduced a
simple model of this energy exchange. This consists of adding an extra
degree of freedom $s$ and rescaling momentum by $s$:
\begin{align}
  \label{eq:nose}
  F &= H(q,p s^{-1}) + N(s,p_s), &&& \text{where }N(s,p_s) =\dfrac{1}{2 M} p_s^2 + n k T \ln s,
\end{align}
where $n$ is the number of degrees of freedom of the system $S$, $M$
is the ``mass'' of the thermostat and $k$ is Boltzmann's
constant. Nos{é}'s thermostated Hamiltonian $F$ has two desirable
properties: the orbit average of twice the kinetic energy,
$\temp{} = \norm{p s^{-1}} ²$, is $T$ and the thermostated system is
Hamiltonian.

Hoover's reduction of Nos{é}'s thermostat eliminates the state
variable $s$ and rescales time $t$~\cite{hoover}:

\[ q = q,\qquad ρ = ps^{-1},\qquad \ddt{τ} = s \ddt{t},\qquad ξ = \ddt[s]{τ}. \]

The Nos{é}-Hoover thermostated simple harmonic oscillator reduces to
the following ``simple'' system:
\begin{align}
  \label{eq:nose-hoover}
  \dot{q} &= ρ, && \dot{ρ} = -q - \xi ρ, &&
  \dot{\xi} = \left( ρ ² - T \right)/M.
\end{align}
Legoll, Luskin and Moeckel show in~\cite{MR2299758} that near the
decoupled limit of $M=∞$ and $ξ=0$, the thermostated harmonic
oscillator \eqref{eq:nose-hoover} is non-ergodic. This is done via
averaging, which reduces the thermostated equations to a
non-degenerate twist map, that shows the existence of KAM tori. In a
subsequent paper, the authors extend the averaging argument to
thermostated integrable $n$-degree of freedom
Hamiltonians~\cite{MR2519685}. It is shown that the averaged equations
are integrable, which implies that near the decoupled limit the
thermostated system's orbits remain close to these invariant tori over
a long, but finite, time horizon. This is a significantly weaker
result than the $1$-degree of freedom result.

\subsection{Complete integrability \& KAM sufficiency}
\label{sec:kam-suff}

Let $θ ∈ \T^n = \R^n/\Z^n$ and $I ∈ \R^n \equiv T^*_{θ} \T^n$ be
coordinates on the cotangent bundle of $\T^n$ equipped with its
canonical symplectic structure. Let $H_{ε}(θ,I) = H_0(I) + ε
H_1(θ,I;ε)$ be a parameterized family of Hamiltonians that is $C^r$ in
all its variables where $r>2n$. For $ε = 0$, the Hamiltonian is
completely integrable with invariant tori $\T^n × \set{I=C}$ and the
flow on these invariant tori is a translation-type flow with frequency
vector ${ω}_0 = \left[ ∂ H_0 / ∂ I_i \right]$. The Hamiltonian $H_0$
(or, by abuse of terminology, $H_{ε}$) is said to be \defn{Kolmogorov
non-degenerate} at $I=C$ if~({\cf} \cite[\S 2.1]{MR2554208}, \cite[\S
1.2]{MR2221801}):
\begin{align}
  \label{eq:kolm-non-deg}
  \det \left. d {ω}_0 \right|_{I=C} &\neq 0, &&&\text{ where } d {ω}_0 &= \left[ \frac{ ∂ ² H_0 }{ ∂ I_i ∂ I_j } \right]
  \intertext{ or it is \defn{iso-energetically non-degenerate} at $I=C$ if: }
  \label{eq:kolm-iso-non-deg}
  \det \begin{bmatrix} d {ω}_0 & {ω}_0 \\ {ω}_0' & 0 \end{bmatrix}_{I=C} &\neq 0, &&& \text{ where } {ω}_0 &= \left[ \frac{ ∂ H_0 }{ ∂ I_i } \right].
\end{align}

If $H_0$ is Kolmogorov non-degenerate, then, for all $ε$ sufficiently
small, there is an open neighbourhood $W ⊂ \R^n$ containing $C$, a
measurable subset $W_{ε} ⊂ W$ and a symplectic diffeomorphism on $\T^n × W$ that
conjugates the Hamiltonian flow of $H_{ε=0}$ on $\T^n × W_{ε}$ to that
of $H_{ε}$. The set $W_{ε}$ has a complement of measure $O(ε)$ and the
conjugacy is as smooth as $H_{ε}$. In particular, the set of invariant
tori that are ``preserved'' under perturbation has positive measure,
which precludes ergodicity.

If $H_0$ is iso-energetically non-degenerate, then similar conclusions
hold: for all $ε$ sufficiently small, there is an open neighbourhood
$W ⊂ H_0^{-1}(c)$, $c=H_0(C)$, containing $C$, a measurable subset $W_{ε} ⊂ W$
and an energy-preserving diffeomorphism $H_{ε=0}^{-1}(c) →
H_{ε}^{-1}(c)$ that conjugates the Hamiltonian flow of $H_{ε=0}$
restricted to $\T^n × W_{ε} ⊂ H_{ε=0}^{-1}(c)$ to a time-change of the
Hamiltonian flow of $H_{ε}$ restricted to $H_{ε}^{-1}(c)$. The
complement of $W_{ε}$ in $W$ has measure $O(ε)$, which precludes
ergodicity on the energy level $c$.

These two forms of non-degeneracy are independent, see e.g. \cite[\S
1.2]{MR2221801}.

\subsection{The Thermostated Free Particle}
\label{sec:tfp}

In Andersen's paper on his barostat, a simplifying assumption is made:
the mechanical system is periodic, or in other words, the
configuration space is a torus~\cite{andersen}. The present paper
starts by proving the following results about thermostated free
particles on flat tori. Specifically, let $\T^n = \R^n/\Z^n$ be the
$n$-dimensional torus and say that the \emph{set of thermostatic
  equilibria} of the thermostat $F$ \eqref{eq:nose} is the set of
points where $\dot{s} = 0 = \dot{p}_s$ (it is not assumed this set is
invariant).

\begin{theorem}
  \label{thm:kam-tori-free-particle}
  Let $V : \T^n → \R$ be $C^r$ for some $r>2n+2$, $⎸ · ⎹ ²$ a flat
  Riemannian metric on $\T^n$ and let $H_{ε} : \cotangent \T^n → \R$ be
  the family of mechanical Hamiltonians
  \begin{equation}
    \label{eq:h-eps}
    H_{ε}(q,p) = ½ ⎸ p ⎹ ² + ε V(q).
  \end{equation}
  Fix the thermostat mass $M > 0$, temperature $T>0$ and let $F_{ε}$
  be the Nos{é}-thermostated Hamiltonian~\eqref{eq:nose} associated to
  $H_{ε}$. There is an open neighbourhood of the set of thermostatic
  equilibria of $F_{0}$ on which $F_{0}$ is both Kolmogorov and
  iso-energetically non-degenerate.
\end{theorem}

The Hamiltonian $H_{0}$ is purely kinetic, so the Nos{é}-thermostated
Hamiltonian $F_{0}$ can be thought of as describing the thermostated
free particle on a flat $n$-torus; in this case, the set of
thermostatic equilibria \emph{is} invariant. Moreover, by the real
analyticity of $F_{0}$ and consequently the real analyticity of the
change of coordinates to action-angle variables, each open set on
which the non-degeneracy condition holds is dense in the extended
phase space of the thermostated free particle. It follows from these
facts and the remarks about KAM theory in subsection
\ref{sec:kam-suff}, that for all $ε$ sufficiently small, the
Hamiltonian flow of the Nos{é} thermostat $F_{ε}$ is not ergodic on
any energy level. The proof of Theorem
\ref{thm:kam-tori-free-particle} appears in section
\ref{sec:const-temp-thermostat}.

\subsection{The high-temperature limit}
\label{sec:high-t-limit}

Theorem~\ref{thm:kam-tori-free-particle} allows us to investigate the
dynamics of Nos{é}'s thermostat near the high-temperature limit $T=∞$
with the thermostat mass $M$ held constant. It proves non-degeneracy,
in both of the above senses, of a suitably rescaled thermostat at the
$T=∞$ limit, thereby establishing the existence of positive measure
sets of KAM tori for thermostated mechanical systems on the $n$-torus
at sufficiently high temperatures.

\begin{theorem}
  \label{thm:kam-tori-high-temperature-limit}
  Let $V : \T^n → \R$ be $C^r$ for some $r>2n+2$, $⎸ · ⎹ ²$ a flat
  Riemannian metric on $\T^n$ and let $H : \cotangent \T^n → \R$ be
  the mechanical Hamiltonian
  \begin{equation}
    \label{eq:h}
    H(q,p) = ½ ⎸ p ⎹ ² + V(q).
  \end{equation}
  Fix the thermostat mass $M > 0$. The rescaled Nos{é}-thermostated
  Hamiltonian $F$ \eqref{eq:nose} associated to $H$ is both Kolmogorov
  and iso-energetically non-degenerate in the $T=∞$ limit.
\end{theorem}

The rescaled thermostat is explained in more detail in
section~\ref{sec:rescaled-thermostat}. Theorem~\ref{thm:kam-tori-high-temperature-limit}
is, in fact, a corollary of Theorem~\ref{thm:kam-tori-free-particle}
where the role of the small parameter $ε$ is played by the inverse
temperature $β$. The proof of Theorem
\ref{thm:kam-tori-high-temperature-limit} appears in section
\ref{sec:const-temp-thermostat}.

These results are stronger than those obtained in \cite{MR2519685}
largely because multi-dimensional averaging is not needed to prove
Theorem \ref{thm:kam-tori-free-particle}. Instead, we are able to use
the complete integrability of the Nos{é}-thermostated free particle to
verify the Kolmogorov and iso-energetic non-degeneracy conditions. 

A corollary of Theorem \ref{thm:kam-tori-free-particle}
(resp. \ref{thm:kam-tori-high-temperature-limit}) is that if $⎸ · ⎹ ₁
$ is a Riemannian metric that is $C^r$-sufficiently close to the flat
metric $⎸ · ⎹ $, then the conclusion of the Theorem
\ref{thm:kam-tori-free-particle}
(resp. \ref{thm:kam-tori-high-temperature-limit}) holds for the
Nos{é}-thermostated Hamiltonian $F ₁$ associated to $H ₁ (q,p) = ½ ⎸ p
⎹ ₁ ^2$ (resp. $H ₁(q,p) = ½ ⎸ p ⎹_1^2 + V(q)$).

\subsection{Variable-Mass Thermostats}
\label{sec:alt-thermostats}

A natural question that arises in light of the above results on the
existence of invariant tori is whether there are thermostats like
Nos{é}'s that do not possess these invariant KAM tori in the large
temperature limit. A Nos{é}-like, or variable-mass, thermostat is one
which involves momentum rescaling and the thermostatic equilibrium
(where $\dot s = 0 = \dot p_s$) is independent of that
rescaling. \cite[Theorem 1.2]{arxiv-pp1} proves that in such a case
the thermostat $N$ is characterized by a smooth positive function $Ω_T
= Ω_T(s)$, possibly parameterized by the temperature $T$, such that
\begin{equation}
  \label{eq:nose-like-thermostat}
  N = \frac{1}{2}\, Ω_T\, p_s^2 + n k T\, \ln s.
\end{equation}
The function $1/Ω_T$ might be viewed as a variable thermostat mass.

With this characterization of variable-mass thermostats, this paper
proves

\begin{theorem}
  \label{thm:kam-tori-for-nose-like-thermostats-free-particle}
  Assume the Hamiltonian $H_{ε}$ satisfies the hypotheses of Theorem
  \ref{thm:kam-tori-free-particle}, fix $T>0$ and let
  $Ω=Ω_T ∈ C^r(\R^+,\R^+)$ and $N$ be the Nos{é}-like thermostat
  \eqref{eq:nose-like-thermostat}. Then there is an open set, whose
  closure contains the set of thermostatic equilibria of the
  thermostated Hamiltonian $F_0$ \eqref{eq:nose}, on which $F_0$ is both
  Kolmogorov and iso-energetically non-degenerate.
\end{theorem}

The infinite temperature limit corollary of
Theorem~\ref{thm:kam-tori-for-nose-like-thermostats-free-particle} is

\begin{theorem}
  \label{thm:kam-tori-for-nose-like-thermostats}
  Assume the hypotheses of Theorem
  \ref{thm:kam-tori-high-temperature-limit}, let $N$ be the
  Nos{é}-like thermostat \eqref{eq:nose-like-thermostat} and let $Ω_T
  ∈ C^r(\R^+,\R^+)$ be such that $R_T(s)=Ω_T(s/√ T)$ converges to a
  limit $Ω ∈ C^r(\R^+,\R^+)$ as $T → ∞$. Then the thermostated
  Hamiltonian $F_0$ \eqref{eq:nose} with the variable-mass thermostat
  $N$ \eqref{eq:nose-like-thermostat} determined by $Ω_T$ is both
  Kolmogorov and iso-energetically non-degenerate in the $T=∞$ limit.
\end{theorem}

These two theorems are proven in a similar manner to their
constant-mass counterparts. Of course, the constant-mass theorems are
special cases, but because the variable-mass case is more involved, we
have elected to present proofs of each. The proofs appear in section
\ref{sec:high-temp-limit}.

\section{Terminology and Notation}
\label{sec:term-not}

Generating functions are a classical and convenient way to create
canonical transformations. To explain, let $(q',p') = f(q,p)$ be a
canonical transformation, so that $q' · dp' + p · dq = d φ$ is closed
and therefore locally exact. That is, there is a locally-defined
function $φ = φ(p';q)$ of the mixed coordinates $(p';q)$ such that $q'
= ∂ φ/∂ p'$ and $p = ∂ φ/∂ q$. The transformation $f$ is implicitly
determined by $φ$. The identity transformation has the generating
function $φ = q · p'$.

In the sequel, a canonical system of coordinates $(x,X) = (x_1 , …,
x_n, X_1, …, X_n)$ are denoted using the capitalization convention:
the Liouville $1$-form equals $\sum_{i=1}^n X_i \D x_i$ and $X_i$ is
the momentum conjugate to the coordinate $x_i$.

In practice, construction of action-angle coordinates for a particular
Hamiltonian is a very difficult problem. However, \textem{approximate}
action-angle coordinates may be constructed by methods similar to
their construction in the Birkhoff normal form: by means of a sequence
of generating functions that transform the Hamiltonian into a
near-integrable form. In this case, one verifies non-degeneracy for
the integrable approximation.

\section{The Rescaled Thermostat}
\label{sec:rescaled-thermostat}

The parameters $n, k, T$ that enter the thermostat may have physical
significance, but for the purposes here, it suffices to amalgamate
$nkT$ into a single parameter, which is denoted by $T$. Let us rescale
the variables in the Nos{é} thermostat
\begin{align}
\label{eq:rescaling}
q  &=  \sqrt{M}\, w \bmod \Z^n, &&& p  &=  W / \sqrt{M}, &&& s  &=  σ/\sqrt{MT}, &&& p_s  &=  \sqrt{MT}\, Σ.
\end{align}
With this canonical change of variables, the thermostated Hamiltonian
for $H$ \eqref{eq:h} is, where $β = 1/T$,
\begin{equation}
  \label{eq:rescaled-f}
  F = T \times \underbrace{\left[ ½ ⎸ W/σ ⎹  ² + ½ Σ ² + β V(w √ M) + \ln σ \right]}_{F_{β}} - ½ T \ln(MT).
\end{equation}
The coordinates $(w,σ)$ and $(W,Σ)$ are canonically conjugate, so up
to a constant rescaling of time, the Hamiltonian flow of $F$ equals
that of $F_{β}$. Additionally, when $M=1$, $F_{β}$ is the
Nos{é}-thermostated Hamiltonian for the Hamiltonian $H_{β}$
\eqref{eq:h-eps} of Theorem \ref{thm:kam-tori-free-particle} with
$ε=β$.

Because the thermostat mass, $M$, is fixed and enters into the rescaled
thermostated Hamiltonian $F_{β}$ only through the bounded potential
$V$, the convention is adopted that
\begin{equation}
  \label{eq:eps=1}
  M = 1.
\end{equation}
The analysis below is altered in insignificant ways by this additional
hypothesis.

\section{The Nos{é} Thermostat}
\label{sec:kam}

In the $1$-degree of freedom case, the following lemma illuminates the
nature of the Hamiltonian $F_{β}$ \eqref{eq:rescaled-f}.

\begin{lemma}[\cite{arxiv-pp1}]
  \label{rot-inv}
  Let $n=1$ and $β = 0$. Under the canonical change of coordinates induced by introducing cartesian coordinates,
  $(a,b) = ( σ \cos w, σ \sin w ),$
  the rescaled thermostated Hamiltonian equals
  \begin{equation}
    \label{eq:f0}
    F_0 = ½ \left[ A ² + B ² \right] + ½ \ln \left( a ² + b ² \right).
  \end{equation}
  That is, $F_0$ is a mechanical hamiltonian with a rotationally
  invariant logarithmic potential.
\end{lemma}

In the $n$-degree of freedom case, the Hamiltonian $F_{β}$ is the
average of $n$ such Hamiltonians, under the constraint that the
radial distance from the origin ($σ$) and radial momentum ($Σ$)
coincide for each.

The Hamiltonian $F_0$ is completely integrable. In particular, there
is a family of normally elliptic invariant tori of $F_0$ along the
variety $\set{ σ=⎸ W ⎹ ≠ 0, Σ=0} $, parameterized by the angular
momentum integral $μ = W$. To deduce the existence of a
positive-measure set of invariant KAM tori, one would like to apply a
theorem of R{ü}ssmann and Sevryuk \cite{MR1354540,MR1390625}. The
potential functions $U(σ) = σ^{α}/α$, and the degeneration $U=\ln$,
frustrate this.

Instead of this straightforward route, we are obliged to compute
action-angle variables for $F_0$ and verify non-degeneracy
directly. The computation of action-angle variables is a difficult
problem for most integrable systems, due to the problems involved in
computing the necessary integrals (quadratures). However, in order to
apply KAM theory, it suffices to compute approximate action-angle
variables, or rather action-angle variables for an approximation to
the given integrable Hamiltonian--this is most commonly done in a
neighbourhood of an elliptic critical point via the Birkhoff normal
form. We pursue a similar strategy in a neighbourhood of the
above-mentioned normally elliptic invariant tori. As in the case of
the Birkhoff normal form, one postulates the form of the Hamiltonian
in action-angle variables up to a given order and attempts to solve
for the generating function of the canonical change of variables. This
is the strategy of the following lemma.

\subsection{Constant Temperature Thermostat}
\label{sec:const-temp-thermostat}

In order to state the following lemma, let us recall that a flat
Riemannian metric on the $n$-torus induces an inner product on $\R^n$
which will be denoted by $\ip{}{}$. By means of this inner product,
$\R^n$ and its dual vector space are identified. The identification
will be denoted by $v → v' = \ip v {·} $. The dual inner product on
the dual of $\R^n$ will be denoted by $\ip{}{}$, also, so the inverse
of the map $v → v'$ is $v'' → v$.

The point $P ∈ \R^+ × \T^n × \R × \R^n \equiv \cotangent
\left( \R^+ ×  \T^n \right)$ is denoted by $P=(σ,w,Σ,W)$, which gives
a coordinate system. Similarly, a point $Q ∈ \cotangent \left( \T^1 ×
  \T^n \right)$ is denoted by $Q=(θ,η,I,J)$. The zero section of
$\cotangent X$ is denoted by $Z(X)$.

\begin{lemma}
  \label{mu-is-1}
  Let $C$ be a unit co-vector and let $λ ⊂ \cotangent (\R^+ × \T^n)$
  be the isotropic graph of $w → (1,w,0,C)$. There are neighborhoods
  ${\mathcal O} ⊃ λ$ and ${\mathcal P} ⊃ Z(\T^1 × \T^n)$ and a
  canonical transformation
  $Φ : {\mathcal P} → {\mathcal O} - λ$,
  $(σ, w, Σ, W)= Φ(θ,η,I,J)$, that transforms the Hamiltonian $F_0$
  \eqref{eq:rescaled-f} to the sum $G_0+G_1$ where
  \begin{align}
    \label{eq:bnf-g0}
    %% G0 = I*(1-(11*I)/24+C*J+((3*C^2-1)*J^2)/2)
    %% ln(|C-V|) = 1/2*ln(|C-V|²) = ((-2*C^4)+2*C^2-1/4)*J^4+(C-(4*C^3)/3)*J^3+(1/2-C^2)*J^2-C*J
    G_0 &=
    I \left( 1 - \frac{11}{24} I + \ip C J + \frac{3}{2}  \ip C J ² - \frac{1}{2}⎸ J ⎹ ² \right) + O(5)\\
    \label{eq:bnf-g1}
    G_1 &=
    - \ip C J - \ip{C}{J} ² + \frac{1}{2} ⎸ J ⎹ ² + (⎸ J ⎹ ² - \frac{4}{3} \ip{C}{J} ²) \ip{C}{J} \\\notag
    &\phantom{=\ }- \frac{1}{4} ⎸ J ⎹ ⁴ + 2 \ip{C}{J} ²  ⎸ J ⎹ ² - 2 \ip{C}{J} ⁴ +
    O(5)
  \end{align}
  where $I$ has degree $2$, $J$ has degree $1$ and $O(5)$ is a
  remainder term containing terms of degree $≥ 5$.
\end{lemma}

\begin{remark}
  A few points on the statement of Lemma \ref{mu-is-1}:
  \begin{enumerate}
  \item The variables $(θ,η,I,J)$ are angle-action variables, to
    leading order, for the Hamiltonian $F_0$.
  \item It may seem odd that $I$ and $J$ do not both have degree $1$
    in the expansions. The reason is revealed in the proof: the
    expressions in eq.s \ref{eq:bnf-g0}, \ref{eq:bnf-g1} originate
    from degree $4$ Taylor polynomials in coordinates on
    ${\mathcal O}$ and in these coordinates, $I$ is a quadratic
    function and $J$ is linear. This reflects the fact that the
    regular Liouville tori in ${\mathcal O} - λ$ limit onto the
    isotropic torus $λ$.
  \item The remainder term $O(5)$ is the degree $≥ 5$ remainder term
    in the above-mentioned Taylor polynomial.
  \end{enumerate}
\end{remark}

\begin{proof}
  The generating function $φ(Σ,W;u,v) = (1-u) ⎸ W ⎹  Σ + \ip{C-W}{v}$ induces the
  canonical transformation $(σ,w,Σ,W) = f(u,v,U,V)$ where
  \begin{align}
    \label{eq:fgen}
    σ &= (1-u) ⎸ C-V ⎹, &&& w &= -v - U(1-u)/⎸ C-V ⎹  \bmod \Z^n, \\\notag
    Σ &= U/⎸ C-V ⎹,    &&& W &= C-V.
  \end{align}
  This transforms the Hamiltonian $F_0$ in variables $(σ,w,Σ,W)$ to
  \begin{equation}
    \label{eq:f0-after-f}
    F_0 = \underbrace{½ (1-u)^{-2} + ½ ⎸ C-V ⎹ ^{-2} U ² + \ln(1-u)}_{G_0} + \underbrace{\ln ⎸ C-V ⎹ }_{G_1}.
  \end{equation}
  The symplectic map $f$ is singular along the set $\set{V=C}$ (which
  should be mapped to the zero section of $\cotangent \T^n$,
  $\set{(w,W) \mid W=0}$), and it transforms $\set{u=0, U=0}$ to the
  variety $\set{σ=⎸ W ⎹ ≠ 0, Σ=0}$ which consists of quasi-periodic
  and periodic invariant tori.

  The determination of a further coordinate change is independent of
  the final term in $F_0$, which involves only $V$, so we let $G_0 =
  F_0 - \ln ⎸ C-V ⎹ $ as indicated in \eqref{eq:f0-after-f}. With the
  fourth-order Maclaurin expansion of $G_0$, one obtains
  \begin{equation}
    \label{eq:f0-almost-bnf}
    %% (9*u^4)/4+(5*u^3)/3+u^2+U^2*(2*C^2*V^2-V^2/2+C*V+1/2)
     G_0 = \left(2 \ip C V ² - ½  ⎸ V ⎹ ²  + \ip C V + ½\right) U^2 + \left(\frac{9u^2}{4} + \frac{5u}{3}  + 1\right) u^2  + O(5),
  \end{equation}
  where $O(5)$ is the remainder term that contains terms of degree $5$ and higher.

  Because $G_0$ is independent of $v$, one postulates a second
  symplectic transformation $(x,y,X,Y) → (u,v=y,U,V)$ with the
  generating function
  \begin{equation}
    \label{eq:nu-gen-fun}
    ν = ν(U,V;x,y) = xU + \ip y V + \sum_{i,k,l} ν_{ikl} x^i U^k V^l + O(5),
  \end{equation}
  where $l$ is a multi-index, $V^l = \prod_{i=1}^n V_i^{l_i}$, the degree of the terms in the sum are $3$ or $4$ and $O(5)$ is the remainder of terms of degree at least $5$.
  In addition, a transformed Hamiltonian is postulated
  \begin{equation}
    \label{eq:g0-nf}
    G_0 = \left(x^2 + \frac{X^2}{2} \right)\left(α\left(x^2 + \frac{X^2}{2} \right) + \ip{γY}{Y} + \ip{β}{Y} + 1\right) + O(5),
  \end{equation}
  where $α$ is a scalar, $β$ is a co-vector and $γ$ is a symmetric linear map.
  One solves for the generating function $ν$ and $G_0$ simultaneously, and arrives at
  % ν = V*y+U^3*((-(5*C*V)/9)+(233*x)/288-5/18)
  % +U*((55*x^3)/144+((-(5*C*V)/6)-5/6)*x^2
  % +((5*C^2*V^2)/8-V^2/4+(C*V)/2+1)*x)
  \begin{align} \notag
    {ν} &= \ip y V \\\notag
    &+U \left[
      x \left( 1 + \frac{1}{2} \ip{C}{V} - \frac{1}{4} ⎸ V ⎹ ² + \frac{5}{8} \ip{C}{V} ² \right)
      - \frac{5}{6} x ² \left( 1 + \ip C V \right)
      + \frac{55}{144} x ³
    \right] \\     \label{eq:nu-c1}
    &+ U ³ \left[ -\frac{5}{18} + \frac{233}{288} x - \frac{5}{9} \ip C V \right]
    + O(5),
  \end{align}
  and
  \begin{align}
    \label{eq:g0-c1-params}
    {α} &= -\frac{11}{24}, &&&
    {β} &= C, &&&
    {γ} &= \frac{1}{2} \left( 3 CC' - 1 \right).
  \end{align}
  Note that when $n=1$ and $C=1$, the above results simplify to those found in \cite{arxiv-pp1}.

  To complete the proof, let $I = (x ² + X ²/2)$, $θ$ be the conjugate
  angle ($\bmod 2 π$), and $η = y \bmod \Z^n$, $J = Y$. The
  composition of the sequence of above-defined canonical coordinate
  changes defines the canonical coordinate change $Φ^{-1} : (σ,w,Σ,W) →
  (θ,η,I,J)$. The transformed Hamiltonian $G_0$ \eqref{eq:g0-nf} is
  congruent $\bmod\ O(5)$ to that in \eqref[1]{eq:bnf-g0}.

  The form of the function $G_1$ \eqref{eq:bnf-g1} in these coordinates
  is determined by the Maclaurin expansion of $\frac{1}{2} \ln(1-t)$
  to fourth order, combined with the substitution $t = 2 \ip C J - ⎸ J
  ⎹ ²$. This yields the stated expansion of $G_1$ to $O(5)$.
\end{proof}

\begin{proof}[Proof of Theorem \ref{thm:kam-tori-free-particle}]
  The thermostated Hamiltonian $F_{ε} = F_0 + ε V(w) = F_0 +
  O(ε)$ where $O(ε) = ε V(w)$ is $C^r$ and $\Z^n$-periodic in
  $w$. Under the sequence of canonical transformations in lemma
  \ref{mu-is-1}, $w = -η + ρ(θ,η,I,J) + O(5) \bmod \Z^n$ where $ρ$ is
  a real-analytic map, and $O(5)$ is a remainder in $I,J$. So the
  perturbation in the approximate angle-action variables $(θ,η,I,J)$
  is $C^r$ and $O(ε)$.

  The Hessian of $F_0$ in the action variables $(I,J)$ is
  \begin{align}
    \label{eq:hess-f0-i=0=j}
    A &=
    \begin{bmatrix}
      -\frac{11}{12} & C' \\[1mm] C & 1-2CC'
    \end{bmatrix}
    &&& \textrm{when } I=0,J=0.
  \end{align}
  When restricted to the invariant subspace $V$ spanned by $(1,0)$ and
  $(0,C)$, $A$ is non-singular; and the restriction of $A$ to the
  invariant subspace $V^{\perp}$ is also non-singular. Thus, $F_0$ is
  Kolmogorov non-degenerate in a neighbourhood of the invariant
  isotropic torus $λ$. Since $λ$ is determined by the arbitrary unit
  co-vector $C$, $F_0$ is Kolmogorov non-degenerate in a neighbourhood
  of $\set 1 × \T^n × \set 0 × S^{n-1} ⊂ \cotangent \left( \R^+ × \T^n
  \right)$.

  The proof of iso-energetic non-degeneracy follows from Kolmogorov
  non-degeneracy. When $I=0, J=0$, the derivative of $F_0=G_0+G_1$ is
  $dI-\ip{C}{dJ} \equiv (1,-C)$. Therefore, the determinant of the
  bordered Hessian \eqref{eq:kolm-iso-non-deg} is $-1/12$.
\end{proof}

\subsection{The High-Temperature Limit}
\label{sec:high-temp-limit}

\begin{proof}[Proof of Theorem \ref{thm:kam-tori-high-temperature-limit}]
  The rescaled thermostated Hamiltonian $F_{β} = F_0 + β V(w) = F_0 +
  O(β)$ where $O(β) = β V(w)$ is $C^r$ and $\Z^n$-periodic in $w$
  \eqref{eq:rescaled-f}. Therefore, we can apply
  Theorem~\ref{thm:kam-tori-free-particle} to deduce the result.
\end{proof}

\section{Nos{é}-like Thermostats}
\label{sec:nose-like-thermostats}

This section proves theorems
\ref{thm:kam-tori-for-nose-like-thermostats-free-particle} and
\ref{thm:kam-tori-for-nose-like-thermostats}.

\subsection{Constant Temperature Thermostats}
\label{sec:constant-temp-nose-like}

Assume that the temperature $T>0$ is fixed, so by the rescaling
\eqref{eq:rescaling} it can be assumed $T$ is unity. Without loss of
generality, it can be assumed that the rescaled inverse thermostat
mass $Ω$ maps $1$ to $1$. In this case, the Nos{é}-like variable-mass
thermostat for the Hamiltonian $H_{ε}$ \eqref{eq:h} is
\begin{equation}
  \label{eq:rescaled-f-nose-like-constant-temp}
  F_{ε} = ½ ⎸ W/σ ⎹ ² + ½ Ω(σ)\, Σ ² + ε V(w) + \ln σ.
\end{equation}

The set of thermostatic equilibria $Τ$ of a variable mass thermostat
coincides with the set of thermostatic equilibria for the constant
mass thermostat. Given this, the notation and terminology of lemma
\ref{mu-is-1} are used in the following lemma.

\begin{lemma}
  \label{lem:nose-like-normal-form}
  Let $C$ be a unit co-vector and let $λ ⊂ \cotangent (\R^+ × \T^n)$
  be the isotropic graph of $w → (1,w,0,C)$. Assume that $r>2n+2$, $Ω
  ∈ C^r(\R^+,\R^+)$ and $Ω(σ) = 1 + a(σ-1) + b(σ-1) ²/2 + O((σ-1)
  ³)$. If
  \begin{enumerate}
  \item \label{ii:beta}   $β = (1-a/2) C$;                                                                                            
  \item \label{ii:alpha}  $b = 16 α + 3 a ²/2 - 5 a + 22/3$; and     %% b=(96α+9a^2-30a+44)/6                       
  \item \label{ii:gamma}  $γ = (a-2)/4 + \left( 4 α + a²/2 - 2a + 10/3 \right) CC'$, %%(48*C^2*α+3*C^2*a^2+(3-24*C^2)*a+40*C^2-6)/12,
  \end{enumerate}
  then there are neighborhoods ${\mathcal O} ⊃ λ$ and ${\mathcal P} ⊃
  Z(\T^1 × \T^n)$ and a canonical transformation $Φ : {\mathcal P} →
  {\mathcal O} - λ$, $(σ, w, Σ, W)= Φ(θ,η,I,J)$ that transforms the
  Hamiltonian $F_0$ \eqref{eq:rescaled-f-nose-like-constant-temp} into
  the sum $G_0+G_1$ where
  \begin{align}
    \label{eq:bnf-nl} %% 
    G_0 &= I(1 + α I + \ip{β}{J} + \ip{γ J}{J}) + O(5),
  \end{align}
  $G_1$ is given in \eqref[1]{eq:bnf-g0} and $I$ has degree $2$, $J$ has
  degree $1$ and $O(5)$ is a remainder term containing terms of degree
  $≥ 5$.
\end{lemma}

\begin{remark}
  In the case $a=b=0$, one finds that $α, β$ and $γ$ are determined by
  \eqref[1]{eq:g0-c1-params}. If one sets $a=2$ and $α=-1/3$,
  then $β=0$, $γ = 0$ and $G_0 = I(1-I/3) + O(5)$.
\end{remark}

\begin{proof}[Proof of Lemma \ref{lem:nose-like-normal-form}]
  The proof is similar to that of Lemma \ref{mu-is-1}, so only the
  highlights are sketched.

  One utilizes the change of coordinates in \eqref[2]{eq:fgen}. The Maclaurin
  expansion of $G_0$ in the coordinates $(u,v,U,V)$ is
  (c.f. \eqref[1]{eq:f0-almost-bnf}):
  \begin{align}
    \label{eq:f0-almost-bnf-alt}
    G_0 &= U^2 \left[\frac{b}{4} u ² -  \frac{1}{2} \left( a + (b-a) \ip C V \right) u \right.\\\notag
    &\phantom{= U^2 }+ \frac{1}{4} \left( 8 + b-5a \right) \ip C V ² - \frac{1}{4} ( 2-a ) ⎸ V ⎹ ²
    + \left. \frac{1}{4} ( 2-a ) \ip C V + ½\right] \\\notag
    &+ u^2 \left[ \frac{9u^2}{4} + \frac{5u}{3}  + 1 \right]  + O(5).
    %% (9*u^4)/4+(5*u^3)/3+(b*U^2*u^2)/4+u^2+U^2*((C*b*V)/2-(C*a*V)/2-a/2)*u
    %%             +U^2*((C^2*b*V^2)/4-(5*C^2*a*V^2)/4+(a*V^2)/4+2*C^2*V^2-V^2/2
    %%                                -(C*a*V)/2+C*V+1/2)
  \end{align}
  One postulates a generating function $ν$ and form of $G_0$ as in
  equations \ref{eq:nu-gen-fun}--\ref{eq:g0-nf} and solves for $ν$ and
  the parameters $α, β$ and $γ$. One determines that
  %% ν = V*y+U*((α/2+11/18)*x^3+((5*C*a*V)/12-(5*C*V)/6-5/6)*x^2
  %%                            +(2*C^2*α*V^2+(3*C^2*a^2*V^2)/32-(7*C^2*a*V^2)/8
  %%                                         +(a*V^2)/8+(37*C^2*V^2)/24-V^2/4
  %%                                         -(C*a*V)/4+(C*V)/2+1)
  %%                             *x)
  %%         +U^3*(((-α/4)-a^2/16+25/36)*x-(4*C*α*V)/3-(C*a^2*V)/8+(7*C*a*V)/9
  %%         -(7*C*V)/6+a/12-5/18)
  \begin{align} \notag
    {ν} &= \ip y V \\\notag
    &+U \left[
      \left( 1 + \frac{1}{4} (2-a) \ip C V - \frac{1}{8} (2-a) ⎸ V ⎹ ² + \frac{1}{96} (192α + 9a^2 - 84a + 148) \ip{C}{V} ²\right)
      x
    \right.\\\notag
    &\phantom{+U [+}
    - \left.\left( \frac{5}{6} + \frac{5}{12} (2-a) \ip C V \right)
    x ²
    +\left( \frac{11}{18} + \frac{α}{2} \right)
    x ³ 
    \right] \\\notag
    &+ U ³ \left[
      \frac{1}{36} (3a-10) - \frac{1}{72} (96 α + 9 a^2 - 56 a + 84) \ip C V + \frac{1}{144} (100 - 36 α - 9 a^2) x \right] \\     \label{eq:nu-c1-alt}
    &+ O(5),
  \end{align}
  and $α, β$ and $γ$ are given in terms of $a$ and $b$ by \ref{ii:beta}--\ref{ii:gamma} above.
\end{proof}

\begin{proof}[Proof of Theorem \ref{thm:kam-tori-for-nose-like-thermostats-free-particle}]
  We verify iso-energetic and then Kolmogorov non-degeneracy.
  One computes the derivative, up to $O(4)$, to be
  \begin{align}
    \label{eq:dg0-g1}
    d G_0 &= \left( 1 + 2 α I + \ip{β}{J} + \ip{γ J}{J} \right)\,dI +
    I \ip{ \left( β + 2 γ J \right) }{dJ}
    \\\notag
    d G_1 &= \left(-\ip{J}{J}+4 \ip{C}{J}^2+2 \ip{C}{J}+1\right) \ip{J}{dJ}+\\\notag
    &\phantom{=}\left(4 \ip{C}{J} \ip{J}{J}+\ip{J}{J}-8 \ip{C}{J}^3-4
      \ip{C}{J}^2-2 \ip{C}{J}-1\right) \ip{C}{dJ}
  \end{align}
  and the Hessian of $G_0$ and $G_1$, up to $O(3)$, to be
  \begin{align}
    \label{eq:hess-g0-alt}
    d^2 G_0 &=
    \begin{bmatrix}
      2 α & (β + 2 γ J)' \\ (β + 2 γ J) & 2 I γ
    \end{bmatrix},
    &&& \textrm{ and }
    d^2 G_1 &=
    \begin{bmatrix}
      0 & 0\\0 & Q
    \end{bmatrix}
  \end{align}
  where
  \begin{align}
    \label{eq:hess-g1-alt}
    Q &= \left( 1 - ⎸ J ⎹ ² + 2 \ip C J + 4 \ip{C}{J} ² \right) 1 - 2JJ' \\\notag
    &+\left( 4 ⎸ J ⎹ ² - 8 \ip{C}{J} - 24 \ip{C}{J} ² -2 \right) CC' + \left( 2 + 8 \ip C J \right)\left( CJ' + JC' \right).
  \end{align}
  
  Assume that there exists parameters $a$ and $b$ and a unit co-vector
  $C$ such that $F_0$ does not satisfy the iso-energetic
  non-degeneracy condition in a neighbourhood of $I=0$ and $J=0$.

  Let the bordered Hessian \eqref{eq:kolm-iso-non-deg} of $F_0$,
  $\bmod\, O(3)$, be denoted by $B$. Let $I=0$ and $J=ρ C$ for a real
  scalar $ρ$. Let $W$ be the subspace with orthonormal basis
  $(1,0,0)$, $(0,C,0)$ and $(0,0,1)$. $W$ is invariant $\bmod\,
  O(ρ^3)$ by $B$. A calculation yields
  \begin{align}
    \label{eq:B}
    B | W &=
    \begin{bmatrix}
      2 α             & 2 γ ρ + β           & γ ρ^2 + β ρ + 1\\
      2 γ ρ + β       &   - 3 ρ^2 - 2 ρ - 1 &  - ρ^2 - ρ - 1\\
      γ ρ^2 + β ρ + 1 &  - ρ^2 - ρ - 1      & 0
    \end{bmatrix} + O(ρ^3)
    \\\notag
    \det(B | W) &=
    1 - 2 α - 2 β +
    \left( - 4 γ - 2 β^2 - 4 α + 2\right) ρ \\    \label{eq:detB}
    &+\left( - 6 β γ - 2 γ - β^2 + 2 β - 6 α + 3\right) ρ^2 + O(ρ^3)
  \end{align}
  where we have abused notation and let the scalar $β$ (resp. $γ$)
  denote the inner product of the vector $β$ with $C$ (resp. matrix
  $γ$ with $CC'$). The determinant $\det(B|W)=O(ρ^3)$ with real $α, β$
  and $γ$ iff $α=1/2, β=0$ and $γ=0$. If $β=0$, then lemma
  \ref{lem:nose-like-normal-form} implies that $a=2$. The same lemma
  implies that if $a=2$ and $α=1/2$, then $γ ≠ 0$. Therefore,
  $\det(B|W)$ has, at worst, a quadratic zero at $ρ=0$ (a more
  detailed calculation shows the vanishing is at worst linear).

  On the other hand, a calculation shows that $B|W^{\perp} =
  1+O(ρ)$. Therefore, $B=B(ρ)$ is non-degenerate in a deleted
  neighbourhood of $ρ=0$. Since $F_0$ is $C^r$ the iso-energetic
  non-degeneracy condition holds in an open set that contains the
  torus $λ_C$ in its closure. Since $C$ is arbitrary, there is an open
  set on which the iso-energetic non-degeneracy condition holds and
  this set contains the set of thermostatic equilibria in its
  closure.

  Let us now assume that there exist parameters $a, b$ and a unit
  co-vector $C$ such that $F_0$ fails to be Kolmogorov non-degenerate
  on a neighbourhood of $I=0, J=0$.

  Let $A$ be the Hessian of $F_0$, $\bmod\,O(3)$, and let $V$ be the
  subspace spanned by $(1,0)$ and $(0,C)$. When $J = ρ C$ for some
  real scalar $ρ$, then the subspace $V$ is $A$-invariant
  $\bmod\,O(3)$. In this case, $A|V$ is the upper left $2 × 2$ corner
  of $B|W$ \eqref{eq:B}, so
  \begin{equation}
    \label{eq:detA}
    \det \left( A|V \right) = -2α-β^2+\left(-4βγ-4α\right)ρ+\left(-4γ^2-6α\right)ρ^2+O(ρ^3).
  \end{equation}
  %% $A | V^{\perp} = B | W^{\perp} = ( 1 + 2 ρ + 3 ρ ²)$
  On the other hand, $A|V^{\perp}=1+O(ρ)$, so if $A$ is singular
  in a neighbourhood of $I=0, ρ=0$, then $\det \left( A|V
  \right)=O(ρ^3)$ so $α=-β^2/2$ and $γ=-β/2$ and $γ^2=-3 α/2$ (where
  we use the same abuse of notation as we did above). The only
  solution is $α=β=γ=0$.

  When $α, β$ and $γ$ are determined by Lemma
  \ref{lem:nose-like-normal-form}, then $α=β=0$ implies that $a=2$ and
  $γ ≠ 0$. Therefore, $A$ is non-degenerate in some deleted
  neighbourhood of $ρ=0$. This proves that, for any choice of $a$, $b$
  and unit co-vector $C$, any neighbourhood of the isotropic torus
  $λ=λ_C ⊂ \cotangent \left( \R^+ × \T^n \right)$ contains points
  where $F_0$ is Kolmogorov non-degenerate.
\end{proof}

\begin{remark}
  \label{rmk:both-degen}
  It follows from the above proof that if $F_0$ is both Kolmogorov and
  iso-energetic degenerate at $I=0,J=0$, then $a=0$, $b=-8$ and so
  $α=-1/2$, $β=C$ and $γ=-1/2+4/3 CC'$. Moreover, $F_0$ is degenerate
  in both senses along the entire set of thermostatic equilibria. On
  the other hand, if $Ω(σ) ≠ -4(σ-1)^2 + O((σ-1)^3)$, then $F_0$ is
  non-degenerate in one of the two senses along the set of
  thermostatic equilibria.
\end{remark}

\subsection{The high-temperature limit}
\label{sec:high-temp-limit-nose-like}

Let $Ω_T$ be the unscaled inverse thermostat mass as assumed in the
statement of Theorem~\ref{thm:kam-tori-for-nose-like-thermostats}. By
means of the rescaling \eqref{eq:rescaling} with $M=1$, the
thermostated Hamiltonian is transformed to
\begin{equation}
  \label{eq:rescaled-f-nose-like}
  F = T \times \underbrace{\left[ ½  ⎸ W/σ ⎹  ² + ½ R_{T}(σ)\, Σ ² + β V(w) + \ln σ \right]}_{F_{β}} - ½ T \ln(T),
\end{equation}
where $R_T(σ) = Ω_T(σ/\sqrt{T})$. By the hypotheses of Theorem
\ref{thm:kam-tori-for-nose-like-thermostats}, $Ω = \lim_{T → ∞}
R_{T,1}$ exists in $C^r(\R^+,\R^+)$. The proof of
Theorem~\ref{thm:kam-tori-for-nose-like-thermostats} now follows from
Theorem~\ref{thm:kam-tori-for-nose-like-thermostats-free-particle}.

\section{Conclusion}
\label{sec:conclusion}

This paper has shown that $n$-degree of freedom Nos{é}-like
thermostats ``suffer'' from persistence of invariant tori near
suitable completely integrable limits, extending the results of
\cite{arxiv-pp1} which deals with the $n=1$ case. A central role is
played here by the flat (or free) limit and hence the assumption that
the configuration space is an $n$-torus. It remains unclear if these
results extend to other configuration spaces, such as the $n$-sphere
with a round metric; possibly not. More likely is that the results do
extend to $n$-dimensional ellipsoids with distinct axes, or more
generally, Liouville metrics on spheres or products of spheres.

Of equal interest is to study how (topological) entropy or Arnol$'$d
diffusion is generated in these thermostats near the infinite
temperature limit.

\begin{arma}
\section{Compliance with Ethical Standards}
\label{sec:compliance}
Conflict of Interest: The author declares that he has no conflict
of interest.
\end{arma}

\bibliographystyle{amsplain}
\bibliography{nhmd}
\end{document}

%% file: definitions.tex
\usepackage[citebordercolor={0.7 0.7 0.7}]{hyperref}
\usepackage{comment,graphicx,color}
\excludecomment{maximacode}
\excludecomment{arxivabstract}
\includecomment{arma}
\usepackage{ltbunicode}

\ifcsname showkeystrue\endcsname
\usepackage[notcite,notref]{showkeys}

\fi

\newcommand{\T}{{\bf T}}
\newcommand{\R}{{\bf R}}

\newcommand{\Z}{{\bf Z}}

\newcommand{\set}[1]{\left\{#1\right\}}

\newtheorem{theorem}{Theorem}[section]

\newtheorem{lemma}{Lemma}[section]
\theoremstyle{definition}

\theoremstyle{remark}

\newtheorem{remark}{Remark}[section]

\newcommand{\safeinput}[1]{\IfFileExists{#1}{{\catcode`\&=0\input{#1}}}{\message{File
        #1 does not exists. Skipping.}}}

\newcommand{\D}[1]{\mathrm{d}#1}

\newcommand{\textem}[1]{{\em #1}}
\newcommand{\defn}[1]{\textem{#1}}

\newcommand{\cf}[0]{{\it c.f.}}

\ifcsname nhpaper\endcsname
\else
\fi

\newcommand{\temp}[1]{2{\mathrm K}}
\newcommand{\norm}[1]{⎸ #1 ⎹}
\newcommand{\ip}[2]{\langle\hspace{-0.5ex}\langle #1, #2 \rangle\hspace{-0.5ex}\rangle}

\newcommand{\ddt}[2][\ ]{\frac{\D{#1}}{\D{#2}}}

\newcommand{\cotangent}{T^*}
\renewcommand{\eqref}[2][0]{{\ifnum#1=0(\fi}eq. \ref{#2}{\ifnum#1=0)\fi}}

\long\def\aureply#1{\ifhmode\newline\fi\noindent{\bf Author Reply}:\ {\em #1}}

\def\gettimestamp#1<#2>{\def\timestamp{\tt #2}}